\documentclass[12pt,twoside]{amsart}
\usepackage{amssymb,amsmath,amsthm, amscd, enumerate, mathrsfs}
\usepackage{graphicx, hhline}
\usepackage[all]{xy}
\usepackage{hyperref}
\hypersetup{colorlinks=true}

\title{Koll\'ar--Nadel type vanishing theorem}
\author{Osamu Fujino} 
\date{2016/7/10, version 0.01}
\address{Department of Mathematics, Graduate School of Science, 
Osaka University, Toyonaka, Osaka 560-0043, Japan}
\email{fujino@math.sci.osaka-u.ac.jp}

\DeclareMathOperator{\Coker}{Coker}

\newtheorem{thm}{Theorem}[section]
\theoremstyle{definition}
\newtheorem*{rem}{\it{Remark}}
\newtheorem*{ack}{Acknowledgments} 

\makeatletter
    
    \@addtoreset{equation}{section}
\makeatother
\begin{document}

\maketitle 

\begin{abstract}
We prove an analytic generalization of Koll\'ar's 
vanishing theorem, which contains the Nadel vanishing 
theorem as a special case. 
\end{abstract}

\section{Introduction}

In the conference, I talked about 
the Hodge theoretic aspect of injectivity and 
vanishing theorems (see \cite{fujino1}, \cite{fujino2}, 
and \cite{fujino3}). 
Here, I will explain some analytic generalizations. 
In \cite{fujino-matsumura}, Shin-ichi Matsumura and I 
established the 
following theorems. 

\begin{thm}[{\cite[Theorem A]{fujino-matsumura}}]\label{f-thm1}
Let $F$ be a holomorphic line bundle 
on a compact K\"ahler manifold $X$ and let $h$ be a singular 
hermitian metric on $F$. 
Let $M$ be a holomorphic line bundle on $X$ 
equipped with a smooth hermitian metric $h_M$. 
We assume that 
$$
\sqrt{-1}\Theta_{h_M}(M)\geq 0 \quad 
\text{and}\quad \sqrt{-1}\Theta_h(F)-a\sqrt{-1}\Theta_{h_M}(M)\geq 0
$$ for some $a>0$. 
Let $s$ be a nonzero global section of $M$. 
Then the map 
$$
\times s: H^i(X, \omega_X\otimes F\otimes 
\mathcal J(h))\to H^i(X, \omega_X\otimes F\otimes \mathcal J(h)\otimes M) 
$$ 
induced by $\otimes s$ is injective for every $i$, where 
$\omega_X$ is the canonical bundle of $X$ and 
$\mathcal J(h)$ is the multiplier ideal sheaf of $h$. 
\end{thm}

Theorem \ref{f-thm1} is a generalization of Enoki's injectivity theorem 
(see \cite[Theorem 0.2]{enoki}). 
Although the formulation of Theorem 
\ref{f-thm1} may look artificial, it 
has many interesting applications (see \cite{fujino-matsumura}). 
Theorem \ref{f-thm2} below is a Bertini-type theorem for multiplier ideal 
sheaves. 

\begin{thm}[{\cite[Theorem 1.10]{fujino-matsumura}}]\label{f-thm2}
Let $X$ be a compact complex manifold, let $\Lambda$ be a free 
linear system on $X$ with $\dim \Lambda\geq 1$, and 
let $\varphi$ be a quasi-plurisubharmonic 
function on $X$. 
We put 
$$
\mathcal G=\{ H\in \Lambda \, |\, {\text{$H$ is smooth 
and $\mathcal J(\varphi|_H)=\mathcal J(\varphi)|_H$}}\}.  
$$ Then $\mathcal G$ is dense in $\Lambda$ in the classical topology. 
Note that $\mathcal J (\varphi)$ is the multiplier ideal 
sheaf of $\varphi$. 
\end{thm}

The main purpose of this 
paper is to prove the following theorem, which is a 
slight generalization of \cite[Theorem D]{fujino-matsumura}, 
as an application of Theorem \ref{f-thm1} and Theorem \ref{f-thm2}. 

\begin{thm}[Vanishing theorem of Koll\'ar--Nadel type]\label{f-thm3} 
Let $f:X\to Y$ be a holomorphic map from 
a compact K\"ahler manifold $X$ to 
a projective variety $Y$. 
Let $F$ be a holomorphic 
line bundle on $X$ equipped with 
a singular hermitian metric $h$. 
Let $H$ be an ample line bundle on $Y$. 
Assume that there exists a smooth hermitian metric $g$ on 
$f^*H$ such that 
$$
\sqrt{-1}\Theta_g(f^*H)\geq 0 \quad \text{and}\quad 
\sqrt{-1}\Theta_h(F)-\varepsilon 
\sqrt{-1}\Theta_g(f^*H)\geq 0
$$ 
for some $\varepsilon >0$. 
Then we have 
$$
H^i(Y, R^jf_*(\omega_X\otimes F\otimes \mathcal J(h)))=0
$$ 
for every $i>0$ and $j$, 
where $\omega_X$ is the canonical bundle of $X$ 
and $\mathcal J(h)$ is the multiplier ideal 
sheaf associated to the singular hermitian metric $h$. 
\end{thm}

We can easily see that Theorem \ref{f-thm3} 
contains Demailly's original formulation of the Nadel vanishing 
theorem (see \cite[Theorem 1.4]{fujino-matsumura}) 
and Koll\'ar's vanishing theorem (see \cite[Theorem 2.1 (iii)]{kollar}) 
as special cases.  
Therefore, we call Theorem \ref{f-thm3} the vanishing theorem 
of Koll\'ar--Nadel type. 
For a related vanishing theorem, 
see \cite[Theorem 1.3]{matsumura}. 

In this paper, we will freely use the same notation as in \cite{fujino-matsumura}. 

\section{Proof of Theorem \ref{f-thm3}}

In this section, we prove Theorem \ref{f-thm3} as an application of 
Theorem \ref{f-thm1} and Theorem \ref{f-thm2}. 
I hope that the following proof will show the reader how to use 
Theorem \ref{f-thm1} and Theorem \ref{f-thm2}. 

\begin{proof}[Proof of Theorem \ref{f-thm3}] 
We use the induction on $\dim Y$. 
If $\dim Y=0$, then the statement is obvious. 
We take a sufficiently large positive integer $m$ and a general 
member $B\in |H^{\otimes m}|$ such that 
$D=f^{-1}(B)$ is smooth, 
contains no associated primes of $\mathcal O_X/\mathcal J(h)$, 
and satisfies $\mathcal J(h|_D)=\mathcal 
J(h)|_D$ by Theorem \ref{f-thm2}. 
By the Serre vanishing theorem, we may further assume that 
\begin{equation}\label{f-eq-1}
H^i(Y, R^jf_*(\omega_X\otimes F\otimes \mathcal J(h))
\otimes H^{\otimes m})=0
\end{equation} 
for every $i>0$ and $j$. 
We have the following big commutative diagram. 
$$
\xymatrix{
 & 0\ar[d] & 0\ar[d]&\\
0 \ar[r]& \mathcal J(h) \otimes \mathcal O_X(-D)
\ar[d]
\ar[r]^{\quad \quad  \alpha}
&\mathcal J(h)\ar[d]\ar[r]& \Coker \alpha 
\ar[r]\ar[d]^{\beta}& 0 \\ 
0 \ar[r]& \mathcal O_X(-D)\ar[d]
\ar[r]
&\mathcal O_X \ar[d]\ar[r]& \mathcal O_{D}\ar[r]& 0 \\
&\left(\mathcal O_X/\mathcal J(h)\right)\otimes 
\mathcal O_X(-D) \ar[r]^{\quad \quad \quad \gamma}
\ar[d]& \mathcal O_X/\mathcal J(h)\ar[d]&\\
& 0& 0&
}
$$ 
Since $D$ contains no associated primes of $\mathcal O_X/\mathcal J(h)$, 
$\gamma$ is injective. 
This implies that $\beta$ is injective by the 
snake lemma and that $\Coker \alpha=\mathcal J(h)|_D
=\mathcal J(h|_D)$. 
Thus we obtain the following short exact sequence: 
$$
0\to \mathcal J(h)\otimes \mathcal O_X(-D)\to 
\mathcal J(h)\to \mathcal J(h|_D)\to 0. 
$$ 
By taking $\otimes \omega_X\otimes F\otimes \mathcal O_X(D)$ 
and using adjunction, 
we obtain the short exact sequence:   
$$
0\to \omega_X\otimes F\otimes \mathcal J(h)
\to \omega_X\otimes F\otimes \mathcal J(h) \otimes f^*H^{\otimes m}
\to \omega_D\otimes F|_D\otimes \mathcal J(h|_D)\to 0.  
$$ 
Therefore, we see that 
\begin{equation}\label{f-eq-2}
\begin{split}
0 &\to R^jf_*(\omega_X\otimes F\otimes \mathcal J(h)) 
\to R^jf_*(\omega_X\otimes F\otimes \mathcal J(h))
\otimes H^{\otimes m}\\& 
\to R^jf_*(\omega_D\otimes F|_D\otimes \mathcal J(h|_D))\to 0
\end{split} 
\end{equation} 
is exact 
for every $j$ since $B$ is a general member of $|H^{\otimes m}|$. 
By induction on $\dim Y$, we have 
\begin{equation}\label{f-eq-3}
H^i(B, R^jf_*(\omega_D\otimes F|_D\otimes 
\mathcal J(h|_D)))=0
\end{equation} for 
every $i>0$ and $j$. 
By taking the long exact sequence associated to \eqref{f-eq-2}, 
we obtain 
\begin{equation*}
\begin{split}
H^i(Y, R^jf_*(\omega_X\otimes F\otimes \mathcal J(h)))
=H^i(Y, R^jf_*(\omega_X\otimes F\otimes \mathcal J(h))
\otimes H^{\otimes m})
\end{split}
\end{equation*} 
for every $i\geq 2$ and $j$ by \eqref{f-eq-3}. 
Thus we have 
\begin{equation}\label{f-eq-4}
H^i(Y, R^jf_*(\omega_X\otimes F\otimes \mathcal J(h)))=0
\end{equation}
for every $i\geq 2$ and $j$ by \eqref{f-eq-1}. 
By Leray's spectral sequence and \eqref{f-eq-1} and \eqref{f-eq-4}, 
we have the 
following commutative diagram: 
$$
\xymatrix{
H^1(Y, \mathcal F^j)\ar[d]_\alpha
\ar@{^{(}->}[r]& H^{j+1}(X, 
\omega_X\otimes F\otimes \mathcal J(h))
\ar@{^{(}->}[d]^{\beta}
\\ 
H^1(Y, \mathcal F^j \otimes H^{\otimes m}) \ar@{^{(}->}[r]& H^{j+1}(X, 
\omega_X\otimes F\otimes \mathcal J(h)\otimes 
f^*H^{\otimes m})
}
$$
for every $j$, 
where $\mathcal F^j=R^jf_*(\omega_X\otimes F\otimes \mathcal 
J(h))$. Note that the horizontal arrows are 
injective. 
Since $\beta$ is injective by Theorem \ref{f-thm1}, 
we obtain that $\alpha$ is also injective. 
By \eqref{f-eq-1}, we have 
\begin{equation*}
H^1(Y, R^jf_*(\omega_X\otimes F\otimes \mathcal J(h))
\otimes H^{\otimes m})=0
\end{equation*} 
for every $j$. 
Therefore, we see that $H^1(Y, 
R^jf_*(\omega_X
\otimes F\otimes \mathcal J(h)))=0$ for every $j$. 
Thus we obtain the desired vanishing theorem:~Theorem \ref{f-thm3}. 
\end{proof}

We close this section with a remark on Nakano semipositive 
vector bundles. 

\begin{rem}\label{f-rem2.1}
Let $E$ be a Nakano semipositive 
vector bundle on $X$. 
We can easily see that Theorem \ref{f-thm3} holds 
even when $\omega_X$ is replaced by $\omega_X\otimes E$. 
We leave the details as an exercise for the reader 
(see \cite[Section 6]{fujino-matsumura}). 
\end{rem}

\begin{ack}
I was partially supported by Grant-in-Aid for Young Scientists (A) 24684002, 
Grant-in-Aid for Scientific Research (S) 16H06337, 
and Grant-in-Aid for Scientific Research (B) 16H03925 from JSPS. 
I thank Shin-ichi Matsumura very much whose 
comments made Theorem \ref{f-thm3} 
better than my original formulation. 
\end{ack}


\end{document}